\documentclass[11pt,a4paper]{amsart}
\usepackage{graphicx,multirow,array,amsmath,amssymb}

\newtheorem{theorem}{Theorem}

\newtheorem{lemma}[theorem]{Lemma}

\begin{document}

\title{Domino tilings of the expanded Aztec diamond}

\author[S. Oh]{Seungsang Oh}
\address{Department of Mathematics, Korea University, Seoul 02841, Korea}
\email{seungsang@korea.ac.kr}

\thanks{Mathematics Subject Classification 2010: 05A15, 05B45, 05B50}
\thanks{This work was supported by the National Research Foundation of Korea(NRF) grant funded by the Korea government(MSIP) (No. NRF-2017R1A2B2007216).}

\maketitle

\begin{abstract}
The expanded Aztec diamond is a generalized version of the Aztec diamond,
with an arbitrary number of long columns and long rows in the middle. 
In this paper, we count the number of domino tilings of the expanded Aztec diamond. 
The exact number of domino tilings is given by recurrence relations of state matrices 
by virtue of the state matrix recursion algorithm, 
recently developed by the author to solve various two-dimensional regular lattice model enumeration problems.
\end{abstract}

\section{Introduction} \label{sec:intro}

In both combinatorial mathematics and statistical mechanics,
domino tiling of the Aztec diamond is an important subject.
The Aztec diamond of order $n$ consists of all lattice squares 
that lie completely inside the diamond shaped region $\{ (x,y) : |x|+|y| \leq n+1 \}$. 
The Aztec diamond theorem from the excellent article of Elkies, Kuperberg, Larsen and Propp~\cite{EKLP}
states that the number of domino tilings of the Aztec diamond of order $n$ is equal to $2^{n(n+1)/2}$.
From the statistical mechanics viewpoint, tilings of large Aztec diamonds exhibit a striking feature. 
The Arctic circle theorem proved by Jockusch, Propp and Shor~\cite{JPS}
says that a random domino tiling of a large Aztec diamond tends to be frozen outside a certain circle.

\begin{figure}[h]
\includegraphics{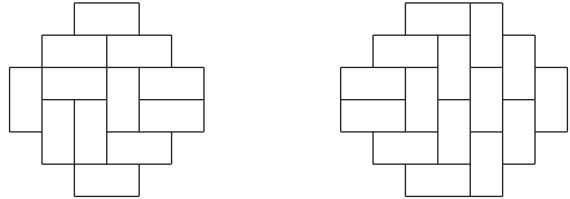}
\caption{The Aztec diamond of order 3, the augmented Aztec diamond of order 3,
and their domino tilings}
\label{fig:Aztec}
\end{figure}

The augmented Aztec diamond looks much like the Aztec diamond,
except that there are three long columns in the middle instead of two.
See Figure~\ref{fig:Aztec}.
The number of domino tilings of the augmented Aztec diamond of order $n$
was computed by Sachs and Zernitz~\cite{SZ}
as $\sum_{k=0}^n {{n}\choose{k}} \cdot {{n+k}\choose{k}}$, known as the Delannoy numbers.
Notice that the former number is much larger than the later.
Indeed, the number of domino tilings of a region is very sensitive to boundary conditions~\cite{MS1, MS2}.
More interesting patterns related to the Aztec diamond allowing some squares removed
have been deeply studied and Propp proposed a survey of these works~\cite{Pr}.

In this paper, we consider a generalized region of the Aztec diamond
which has an arbitrary number of long columns and long rows in the middle. 
The expanded {\em $(p,q)$-Aztec diamond\/} of order $n$, denoted by $AD_{(p,q;n)}$,
is defined as the union of  $2n(n + p + q + 1) + pq$  unit squares,
arranged in bilaterally symmetric fashion as a stack of $2n + q$ rows of squares,
the rows having lengths $p \! + \! 2$, $p \! + \! 4$, \dots, $2n \! + \! p \! - \! 2$,
$2n \! + \! p$, \dots, $2n \! + \! p$, $2n \! + \! p \! - \! 2$, \dots, $p \! + \! 2$,
as drawn in Figure~\ref{fig:domino}.
Let $\alpha_{(p,q;n)}$ denote the number of of domino tilings of $AD_{(p,q;n)}$.
Note that $\alpha_{(p,q;n)} = 0$ for odd $pq$ because $AD_{(p,q;n)}$ consists of odd number of squares.

\begin{figure}[h]
\includegraphics{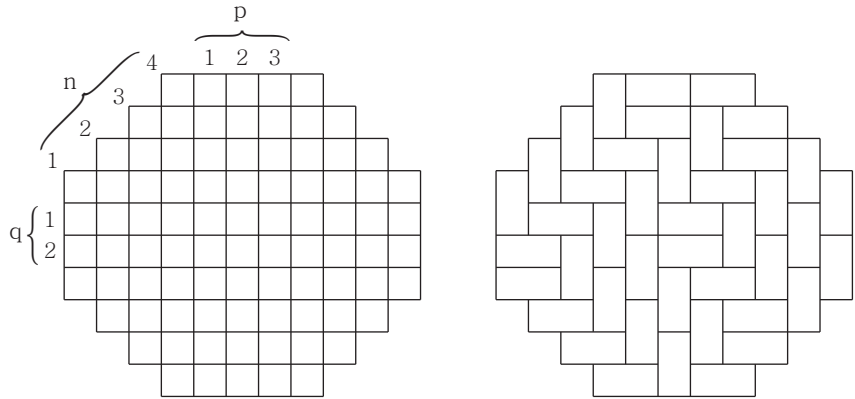}
\caption{The expanded Aztec diamond $AD_{(3,2;4)}$ and a domino tiling}
\label{fig:domino}
\end{figure}

Recently several important enumeration problems regarding various two-dimensional regular lattice models
are solved by means of the {\em state matrix recursion algorithm\/}, introduced by the author.
This algorithm provides recursive matrix-relations to enumerate
monomer and dimer coverings and independent vertex sets known as the Merrifield--Simmons index.
These problems have been major outstanding unsolved combinatorial problems,
and this algorithm shows considerable promise for further two-dimensional lattice model enumeration studies.
See \cite{OhV2, OhD1, OhV1} for more details.

Using the state matrix recursion algorithm,
we present a recursive formula producing the exact number of $\alpha_{(p,q;n)}$.
Throughout the paper, $\mathbb{O}$ is a zero-matrix with an appropriate size,
and $A^t$ is the transpose of a matrix $A$.

\begin{theorem}\label{thm:main}
The number $\alpha_{(p,q;n)}$ of domino tilings of the $(p,q)$-Aztec diamond of order $n$
is the $(1,1)$-entry of the following $2^p \! \times \! 2^p$ matrix
$$\prod_{k=1}^n \begin{bmatrix}
\scriptsize{\begin{bmatrix} A_{p+2k-2} & \mathbb{O} \\ \mathbb{O} & \mathbb{O} \end{bmatrix}}
& A_{p+2k-1} \end{bmatrix}
\cdot \big(C_{p+2n}\big)^q \cdot
\Big(\prod_{k=1}^n \begin{bmatrix}
\scriptsize{\begin{bmatrix} A_{p+2k-2} & \mathbb{O} \\ \mathbb{O} & \mathbb{O} \end{bmatrix}}
& A_{p+2k-1} \end{bmatrix}\Big)^t,$$
where the $2^{k-1} \times 2^k$ matrix $A_k$ and  the $2^k \times 2^k$ matrix $C_k$ are
defined by 
$$A_k = \begin{bmatrix}
\scriptsize{\begin{bmatrix} A_{k-2} & \mathbb{O} \\ \mathbb{O} & \mathbb{O} \end{bmatrix}}
& A_{k-1} \\ A_{k-1} & \mathbb{O} \end{bmatrix} \mbox{ and }
C_k = \begin{bmatrix}
\scriptsize{\begin{bmatrix} C_{k-2} & \mathbb{O} \\ \mathbb{O} & \mathbb{O} \end{bmatrix}}
& C_{k-1} \\ C_{k-1} & \mathbb{O} \end{bmatrix}$$
with seed matrices
$A_1 \! = \! \begin{bmatrix} 0 & 1 \end{bmatrix}$,
$A_2 \! = \! \begin{bmatrix} 1 & 0 & 0 & 1 \\ 0 & 1 & 0 & 0 \end{bmatrix}$,
$C_0 \! = \! \begin{bmatrix} 1 \end{bmatrix}$ and
$C_1 \! = \! \begin{bmatrix} 0 & 1 \\ 1 & 0 \end{bmatrix}$.
Here, $\begin{bmatrix} 1 & 0 \end{bmatrix}$ is used for the undefined matrix
$\scriptsize{\begin{bmatrix} A_0 & \mathbb{O} \\ \mathbb{O} & \mathbb{O} \end{bmatrix}}$
when $p \! = \! 0$ and $k \! = \! 1$.
\end{theorem}
 
We adjust the main scheme of the state matrix recursion algorithm introduced in~\cite{OhD1}
to solve Theorem~\ref{thm:main} in Sections~\ref{sec:stage1}--\ref{sec:stage3} as three stages.

\section{Stage 1. Conversion to domino mosaics} \label{sec:stage1}

First stage is dedicated to the installation of 
the mosaic system for domino tilings on the expanded Aztec diamond region.
A mosaic system was invented by Lomonaco and Kauffman to give
a precise and workable definition of quantum knots representing
an actual physical quantum system~\cite{LK}.
Later, the author {\em et al\/}. have developed a state matrix argument for knot mosaic enumeration
in a series of papers \cite{HLLO, HO, OHLL, OHLLY}.
This argument has been developed further into the state matrix recursion algorithm
by which we enumerate monomer--dimer coverings on the square lattice~\cite{OhD1}.
We follow the notion and terminology in the paper with some modifications.

In this paper, we consider the four {\em mosaic tiles\/}  $T_1$, $T_2$, $T_3$ and $T_4$ 
illustrated in Figure~\ref{fig:tile}.
Their side edges are labeled with two letters a and b as follows:
letter `a' if it is not touched by a thick arc on the tile, and letter `b' for otherwise.

\begin{figure}[h]
\includegraphics{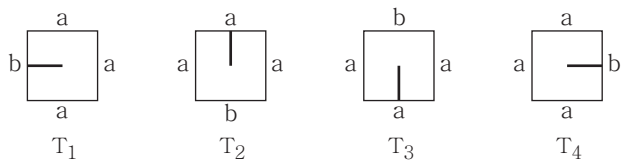}
\caption{Four mosaic tiles labeled with two letters}
\label{fig:tile}
\end{figure}

For non-negative integers $p$, $q$ and $n$,
a {\em $(p,q;n)$-mosaic\/} is an array $M = (M_{ij})$ of those tiles placed on $AD_{(p,q;n)}$,
where $M_{ij}$ denotes the mosaic tile placed at the $i$th column from left to right
and the $j$th row from bottom to top. 
So, it consists of $2n + p$ columns (or $2n + q$ rows) of various length.
We are mainly interested in mosaics whose tiles match each other properly to represent  domino tilings.
For this purpose we consider the following rules. \\

\noindent {\bf Adjacency rule:\/}
Abutting edges of adjacent mosaic tiles in a mosaic are labeled with the same letter.  \\

\noindent {\bf Boundary state requirement:\/}
All boundary edges in a mosaic are labeled with letter a.  \\

As illustrated in Figure~\ref{fig:conversion},
every domino tiling of $AD_{(p,q;n)}$ can be converted into 
a $(p,q;n)$-mosaic which satisfies the two rules.
In this mosaic, $T_1$ and $T_4$ (or, $T_2$ and $T_3$) can be adjoined along the edges
labeled b to produce a dimer.

\begin{figure}[h]
\includegraphics{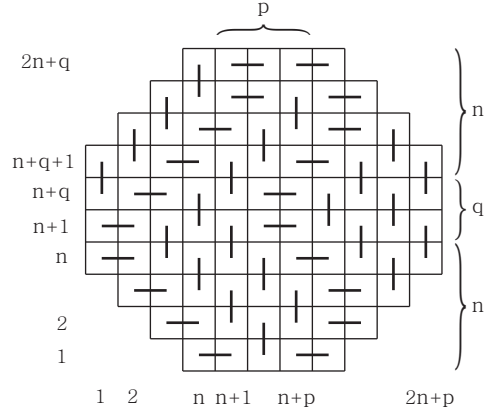}
\caption{Conversion  of the domino tiling of $AD_{(3,2;4)}$ drawn in Figure~\ref{fig:domino} 
to a domino $(3,2;4)$-mosaic}
\label{fig:conversion}
\end{figure}

A mosaic is said to be {\em suitably adjacent\/} if any pair of mosaic tiles
sharing an edge satisfies the adjacency rule.
A suitably adjacent $(p,q;n)$-mosaic is called a {\em domino $(p,q;n)$-mosaic\/}
if it additionally satisfies the boundary state requirement.
The following one-to-one conversion arises naturally. \\

\noindent {\bf One-to-one conversion:\/}
There is a one-to-one correspondence between  
domino tilings of $AD_{(p,q;n)}$ and domino $(p,q;n)$-mosaics.

\section{Stage 2. State matrix recursion formula} \label{sec:stage2}

Now we introduce several types of state matrices for suitably adjacent mosaics.

\subsection{Bar state matrices} 

Consider a suitably adjacent $m \! \times \! 1$-mosaic $M$ for $1 \leq m \leq 2n \! + \! p$,
representing a row of length $m$ in $AD_{(p,q;n)}$, which is call a {\em bar mosaic\/}.
A {\em state\/} is a finite sequence of two letters a and b.
The {\em $l$-state\/} $s_l(M)$ ({\em $r$-state\/} $s_r(M)$) indicates 
the state on the left (right, respectively) boundary edge.
The {\em $b$-state\/} $s_b(M)$ ({\em $t$-state\/} $s_t(M)$) indicates 
the $m$-tuple of states, reading off those on the bottom (top, respectively) 
boundary edges from right to left along the arrows in Figure~\ref{fig:bar}.
The state aa$\cdots$a is called {\em trivial\/}.

\begin{figure}[h]
\includegraphics{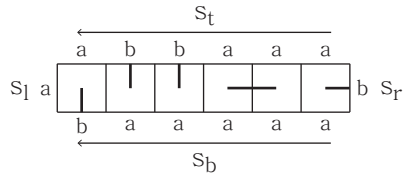}
\caption{A suitably adjacent bar $6 \! \times \! 1$-mosaic with four state indications:
$s_l =$ a, $s_r =$ b, $s_b =$ aaaaab and $s_t =$ aaabba}
\label{fig:bar}
\end{figure}

Length $m$ bar mosaics have possibly $2^m$ kinds of $b$- and $t$-states, called {\em bar states\/}.
We arrange all bar states in the lexicographic order, called the ab{\em -order\/}.
For $1 \leq i \leq 2^m$, let $\epsilon^m_i$ denote the $i$th bar state of length $m$.
 
Given a triple $\langle s_r, s_b, s_t \rangle$ of $r$-, $b$- and $t$-states,
let $\mu_{\langle s_r, s_b, s_t \rangle}$ denote the number of all suitably adjacent bar mosaics $M$ such that 
$s_r(M) = s_r$, $s_b(M) = s_b$, $s_t(M) = s_t$ and $s_l(M)=$ a.
The last triviality condition for $s_l(M)$ is necessary for the left boundary state requirement.

{\em Bar state matrix\/} $X_m$ ($X = A, B$) for suitably adjacent bar $m \! \times \! 1$-mosaics 
is a $2^m \! \times \! 2^m$ matrix $(m_{ij})$ defined by  
$$ m_{ij} = \mu_{\langle \text{x}, \epsilon^m_i, \epsilon^m_j \rangle}$$
where x $\! =\!$ a, b, respectively.
One can observe that information on suitably adjacent bar $m \! \times \! 1$-mosaics is completely encoded 
in two bar state matrices $A_m$ and $B_m$.
Notice that each entry should be either 0 or 1.

\begin{lemma} [Bar state matrix recursion lemma] \label{lem:bar}
Bar state matrices $A_m$ and $B_m$ are obtained by the recurrence relations:
$$A_k = \begin{bmatrix} B_{k-1} & A_{k-1} \\ A_{k-1} & \mathbb{O} \end{bmatrix}
\mbox{ and } B_k = \begin{bmatrix} A_{k-1} & \mathbb{O} \\ \mathbb{O} & \mathbb{O} \end{bmatrix}$$
with seed matrices
$A_1 = \begin{bmatrix} 0 & 1 \\ 1 & 0 \end{bmatrix}$ and
$B_1 = \begin{bmatrix} 1 & 0 \\ 0 & 0 \end{bmatrix}$.
\end{lemma}

Note that we may start with matrices 
$A_0 = \begin{bmatrix} 1 \end{bmatrix}$ and $B_0 = \begin{bmatrix} 0 \end{bmatrix}$
instead of $A_1$ and $B_1$.
This lemma is a simple version of the bar state matrix recursion lemma \cite[Lemma 7]{OhD1},
applying $z=0$.
Here we restate the proof for the completeness of the paper.

\begin{proof}
We use induction on $k$.
A straightforward observation on the three mosaic tiles $T_2$, $T_3$ and $T_4$
establishes the lemma for $k=1$.
For example, $(2,1)$-entry of $A_1$ is 
$$ \mu_{\langle \text{a}, \epsilon^1_2, \epsilon^1_1 \rangle} = 
\mu_{\langle \text{a}, \text{b}, \text{a} \rangle} = 1 $$ 
since only $T_3$ satisfies the requirement.

Assume that bar state matrices $A_{k-1}$ and $B_{k-1}$ satisfy the statement.
Now we consider $A_k$ for one case.
Partition this matrix of size $2^k \! \times \! 2^k$
into four block submatrices of size $2^{(k-1)} \! \times \! 2^{(k-1)}$, 
and consider the 12-submatrix of $A_k$ 
i.e., the $(1,2)$-component in the $2 \! \times \! 2$ array of four blocks.

The entries of the 12-submatrix have the numbers $\mu_{\langle \text{a}, s_b, s_t \rangle}$ 
where $s_b$ and $s_t$ are bar states of length $k$, starting with letters a and b, respectively because of the ab-order.
A suitably adjacent bar $1 \! \times \! k$-mosaic corresponding to these triples 
$\langle \text{a}, s_b, s_t \rangle$ must have unique tile $T_2$ at the rightmost,
and so its second rightmost tile must have $r$-state a by the adjacency rule.
Thus the 12-submatrix of $A_{k}$ is $A_{k-1}$.
Using the same argument, we derive Table~\ref{tab:barset} presenting all possible eight cases as we desired.
This completes the proof.
\end{proof}

\begin{table}[h]
\def\arraystretch{1.2} 
\begin{tabular}{clll}      \hline \hline
 & Submatrix for $\langle s_r, s_b, s_t \rangle$ \ \ & Rightmost tile \ & Submatrix \\    \hline
\multirow{3}{5mm}{$A_k$}
 & 11-submatrix $\langle \text{a}, \text{a} \! \cdot \! \cdot,\text{a} \! \cdot \! \cdot \rangle$ 
 & $T_1$ &  $B_{k-1}$ \\
 & 12-submatrix $\langle \text{a}, \text{a} \! \cdot \! \cdot,\text{b} \! \cdot \! \cdot \rangle$
 & $T_2$ & $A_{k-1}$ \\
 & 21-submatrix $\langle \text{a}, \text{b} \! \cdot \! \cdot,\text{a} \! \cdot \! \cdot \rangle$
 & $T_3$ & $A_{k-1}$ \\    \hline
\multirow{1}{3mm}{$B_k$}
 & 11-submatrix $\langle \text{b}, \text{a} \! \cdot \! \cdot,\text{a} \! \cdot \! \cdot \rangle$ 
 & $T_4$ & $A_{k-1}$ \\    \hline
 & The other four cases & None & $\mathbb{O}$ \\  \hline \hline
\end{tabular}
\vspace{4mm}
\caption{Eight submatrices of $A_k$ and $B_k$}
\label{tab:barset}
\end{table}

\subsection{Three types of bar state matrices} 
Now we categorize each $j$th row of $AD_{(p,q;n)}$ into three types:
lower bar mosaics for $j \! = \! 1, \dots, n$,
central bar mosaics for $j \! = \! n \! + \! 1, \dots, n \! + q$,
and upper bar mosaics for $j \! = \! n \! + \! q \! + \! 1, \dots, 2n \! + q$
as in Figure~\ref{fig:threebars}.
To be a row of a domino $(p,q;n)$-mosaic, each bar mosaic has trivial $l$- and $r$-states a,
and furthermore each lower (or upper) bar mosaic has $b$-state (or $t$-state, respectively)
of the form whose the first and the last letters associated to the rightmost and the leftmost tiles, respectively, are a,
while each central bar mosaic has no further condition.

\begin{figure}[h]
\includegraphics{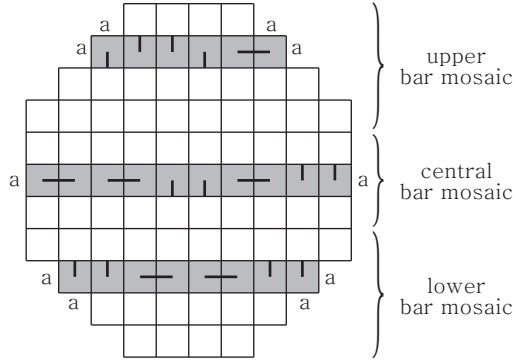}
\caption{Three types of suitably adjacent bar mosaics.}
\label{fig:threebars}
\end{figure}

We refine the definition of a bar state matrix according to this category.
{\em Central bar state matrix\/} $C_m$  for suitably adjacent central bar mosaics of length $m$
is a $2^m \! \times \! 2^m$ matrix $(x_{ij})$ 
where $x_{ij} = \mu_{\langle \text{a}, \epsilon^m_i, \epsilon^m_j \rangle}$.
Obviously $C_m$ is just bar state matrix $A_m$.
{\em Lower bar state matrix\/} $L_m$ for suitably adjacent lower bar mosaics of length $m$
is a $2^{m-2} \! \times \! 2^m$ matrix $(x_{ij})$ 
where $x_{ij} = \mu_{\langle \text{a}, \text{a} \epsilon^{m-2}_i \text{a}, \epsilon^m_j \rangle}$.
Here $\epsilon^{m-2}_i$ in $s_b$ indicates the $i$th $b$-state among $2^{m-2}$ states 
after ignoring the leftmost and the rightmost tiles whose two bottom edges have fixed state a.
Similarly {\em Upper bar state matrix\/} $U_m$ for suitably adjacent upper bar mosaics of length $m$ 
is a $2^m \! \times \! 2^{m-2}$ matrix $(x_{ij})$ 
where $x_{ij} = \mu_{\langle \text{a}, \epsilon^m_i, \text{a} \epsilon^{m-2}_j \text{a} \rangle}$.

\begin{lemma} \label{lem:central}
Central bar state matrix $C_m$ is obtained from the recurrence relation,
for $k = 2, \dots, m$,
$$C_k = \begin{bmatrix}
\scriptsize{\begin{bmatrix} C_{k-2} & \mathbb{O} \\ \mathbb{O} & \mathbb{O} \end{bmatrix}}
& C_{k-1} \\ C_{k-1} & \mathbb{O} \end{bmatrix} $$
starting with
$C_0 = \begin{bmatrix} 1 \end{bmatrix}$ and
$C_1 = \begin{bmatrix} 0 & 1 \\ 1 & 0 \end{bmatrix}$.
\end{lemma}

\begin{proof}
$C_m$ is just bar state matrix $A_m$ in Lemma~\ref{lem:bar}.
Note that the two recurrence relations in the lemma easily merge into one
with new seed matrices $C_0$ and $C_1$.
\end{proof}

\begin{lemma}\label{lem:lower}
The lower bar state matrix is
$$L_m = \begin{bmatrix}
\scriptsize{\begin{bmatrix} A_{m-2} & \mathbb{O} \\ \mathbb{O} & \mathbb{O} \end{bmatrix}}
& A_{m-1} \end{bmatrix}$$
where $A_{m-1}$ and $A_{m-2}$ are obtained from the recurrence relation,
for $k = 3, \dots, m \! - \! 1$,
$$A_k = \begin{bmatrix}
\scriptsize{\begin{bmatrix} A_{k-2} & \mathbb{O} \\ \mathbb{O} & \mathbb{O} \end{bmatrix}}
& A_{k-1} \\ A_{k-1} & \mathbb{O} \end{bmatrix} $$
starting with
$A_1 = \begin{bmatrix} 0 & 1 \end{bmatrix}$ and
$A_2 = \begin{bmatrix} 1 & 0 & 0 & 1 \\ 0 & 1 & 0 & 0 \end{bmatrix}$.
Here, $\begin{bmatrix} 1 & 0 \end{bmatrix}$ is used for the undefined matrix
$\begin{bmatrix} A_0 & 0 \\ 0 & 0 \end{bmatrix}$
when $m=2$.

Furthermore the upper bar state matrix is the transpose of the lower bar state matrix, i.e.,
$$U_m = (L_m)^t.$$
\end{lemma}

\begin{proof}
Lower bar state matrix $L_m$ is slightly differ from the central bar state matrix
so as the first and the last letters of each $m$-tuple $b$-state must be deleted
because of the ignorance of the rightmost and the leftmost tiles, respectively.
This means that we only have to select the rows of $A_m$ in Lemma~\ref{lem:bar}
representing $b$-states whose first and last letters are a.
Therefore we instead use $A_1 = \begin{bmatrix} 0 & 1 \end{bmatrix}$ and
$B_1 = \begin{bmatrix} 1 & 0 \end{bmatrix}$
(according to the ignorance of the last letter associated to the leftmost tile), and
$L_m = \begin{bmatrix} B_{m-1} & A_{m-1} \end{bmatrix}$
(according to the ignorance of the first letter associated to the rightmost tile) only when $k=1,m$.
Note that, when $m=2$,
we have to use $\begin{bmatrix} 1 & 0 \end{bmatrix}$ instead of $\begin{bmatrix} A_0 & 0 \\ 0 & 0 \end{bmatrix}$
since $L_2 = \begin{bmatrix} B_1 & A_1 \end{bmatrix}$.

Upper bar state matrix $U_m$ is just the transpose of $L_m$
because of the symmetricity between a lower bar mosaic and an upper bar mosaic of the same length.
This is just exchanging $b$-states and $t$-states.
\end{proof}

\subsection{State matrices}

{\em State matrix\/} $N_m$ for suitably adjacent mosaics consisting of the $m$ consecutive rows on $AD_{(p,q;n)}$ from the bottom
is a $2^p \! \times \! 2^l$ matrix $(n_{ij})$
where $l$ is the length of the topmost $m$th row (or, if it is one of the upper bar mosaics, $l = m \! + \! 2$)
and $n_{ij}$ is the number of such suitably adjacent mosaics
whose the bottommost bar mosaic has $s_b = \text{a} \epsilon^p_i \text{a}$,
the topmost bar mosaic has $s_t = \epsilon^m_j$
(if it is one of the upper bar mosaics, $s_t = \text{a} \epsilon^m_j \text{a}$),
and all the other boundary edges have state a as the bottom one in Figure~\ref{fig:split}.

\begin{lemma}\label{lem:state}
$$N_{2n+q} = \prod_{k=1}^n L_{p+2k} \cdot \big(C_{p+2n}\big)^q \cdot
\prod_{k=1}^n U_{p+2n+2-2k}$$
\end{lemma}

\begin{proof}
It is enough showing that for $m=1, \dots, 2n \! + \! q$,
$N_m$ is the multiplication of the related $m$ lower, central or upper bar state matrices 
associated to each row.

Use induction on $m$.
Obviously  $N_1 = L_{p+2}$.
Assume that $N_m$ satisfies the statement.
Consider a suitably adjacent mosaic consisting of the $m \! + \! 1$ rows from the bottom.
Split it into a suitably adjacent mosaic consisting of $m$ bar mosaics and a suitably adjacent bar mosaic by tearing off the topmost row.
According to the adjacency rule,
the $t$-state of the lower mosaic and the $b$-state of the topmost bar mosaic
on the abutting horizontal edges must coincide as shown in Figure~\ref{fig:split}.
Let $N_m = (n_{ij})$, $N_{m+1} = (n_{ij}')$ and
the bar state matrix for the topmost bar mosaic be $(a_{ij})$.

\begin{figure}[h]
\includegraphics{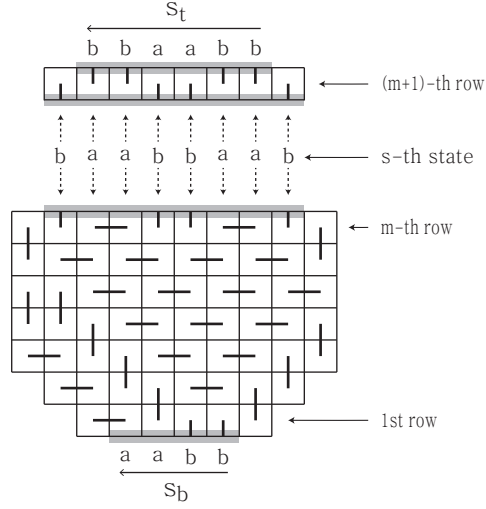}
\caption{Attacing bar mosaics}
\label{fig:split}
\end{figure}

Let $l$ be the length of the abutting horizontal edges, two shaded parts in middle of the picture.
Given an $s$th state among $2^l$ states,
$n_{is}  \cdot  a_{sj}$ indicates the number of suitably adjacent mosaics consisting of $m \! + \! 1$ bar mosaics
where $s_b = \text{a} \epsilon^p_i \text{a}$,
$s_t = \epsilon^m_j$ (or, if it is one of the upper bar mosaics, $s_t = \text{a} \epsilon^m_j \text{a}$),
$m$th bar mosaic has $s_t = \epsilon^m_s$
(or, if it is one of the upper bar mosaics, $s_t = \text{a} \epsilon^m_s \text{a}$),
and all the other boundary edges have state a.
Since all $2^l$ kinds of states arise as states of these $l$ abutting horizontal edges,
we get
$$ n_{ij}' = \sum^{2^l}_{s=1} n_{is}  \cdot a_{sj},$$
so
$N_{m+1}$ is the multiplication of the $m \! + \! 1$ bar state matrices
related to the $m \! + \! 1$ consecutive rows.
\end{proof}

\section{Stage 3. State matrix analyzing} \label{sec:stage3}

\begin{proof}[Proof of Theorem~\ref{thm:main}.]
Each $(1,1)$-entry of $N_{2n+q}$ is the number of suitably adjacent mosaics on $AD_{(p,q;n)}$
such that the bottommost and the topmost bar mosaics have trivial $b$- and $t$-states, respectively,
and all the other boundary edges have state a.
The point is that this satisfies the boundary state requirement to represent domino $(p,q;n)$-mosaics.
Thus we get the equality
$$\alpha_{(p,q;n)} = \mbox{(1,1)-entry of } N_{2n+q}.$$
This combined with Lemmas~\ref{lem:central}--\ref{lem:state}
completes the proof.
\end{proof}


\begin{thebibliography}{11}

\bibitem{EKLP} N. Elkies, G. Kuperberg, M. Larsen and J. Propp,
    {\em Alternating-sign matrices and domino tilings (pats I and II)},
    J. Algebraic Combin. \textbf{1} (1992) 111--132 and 219--234.
\bibitem{HLLO} K. Hong, H. Lee, H. J. Lee and S. Oh,
    {\em Small knot mosaics and partition matrices},
    J. Phys. A: Math. Theor. \textbf{47} (2014) 435201.
\bibitem{HO} K. Hong and S. Oh,
    {\em Enumeration on graph mosaics},
    J. Knot Theory Ramifications \textbf{26} (2017) 1750032.
\bibitem{JPS} W. Jockusch, J. Propp and P. Shor,
    {\em Random domino tilings and the Arctic circle theorem},
    (1998) arXiv:math/9801068.
\bibitem{LK} S. Lomonaco and L. Kauffman,
    {\em Quantum knots and mosaics},
    Quantum Inf. Process. \textbf{7} (2008) 85--115.
\bibitem{MS1} R. Merrifield and H. Simmons,
    {\em Enumeration of structure-sensitive graphical subsets: Theory},
    Proc. Natl. Acad. Sci. USA \textbf{78} (1981) 692--695.
\bibitem{MS2} R. Merrifield and H. Simmons,
    {\em Enumeration of structure-sensitive graphical subsets: Calculations},
    Proc. Natl. Acad. Sci. USA \textbf{78} (1981) 1329--1332.
\bibitem{OhV2} S. Oh,
    {\em Maximal independent sets on a grid graph},
    Discrete Math. \textbf{340} (2017) 2762--2768.
\bibitem{OhD1} S. Oh,
    {\em State matrix recursion algorithm and monomer--dimer problem},
    (preprint).
\bibitem{OHLL} S. Oh, K. Hong, H. Lee and H. J. Lee,
    {\em Quantum knots and the number of knot mosaics},
    Quantum Inf. Process. \textbf{14} (2015) 801--811.
\bibitem{OHLLY} S. Oh, K. Hong, H. Lee, H. J. Lee and M. J. Yeon,
    {\em Period and toroidal knot mosaics},
    J. Knot Theory Ramifications \textbf{26} (2017) 1750031.
\bibitem{OhV1} S. Oh and S. Lee,
    {\em Enumerating independent vertex sets in grid graphs},
    Linear Algebra Appl. \textbf{510} (2016) 192--204.
\bibitem{Pr} J. Propp,
    {\em Enumeration of matchings: problems and progress},
    New Perspectives in Geometric Combinatorics, MSRI Publications \textbf{38} (1999) 255--290.
\bibitem{SZ} H. Sachs and H. Zernitz,
    {\em Remark on the dimer problem},
    Discrete Appl. Math. \textbf{51} (1994) 171--179.

\end{thebibliography}
\end{document}